\documentclass[11pt]{article}
\usepackage[pagebackref,colorlinks=true,urlcolor=blue,linkcolor=blue,citecolor=blue]{hyperref}

\usepackage{amsthm, amssymb, bm}
\usepackage{soul, color}
\usepackage[]{amsmath}
\usepackage[]{amsfonts}
\usepackage[]{fancyhdr}
\usepackage[]{graphicx}
\graphicspath{{/EPSF/}{../figures/}{figures/}}
\newcommand{\tred}[1]{{\color{red}{#1}}}

\newtheorem{theorem}{Theorem}[]
\newtheorem{lemma}[theorem]{Lemma}

\newtheorem{remark}[theorem]{Remark}

\def \Cm {\mathbb{C}}

\def \Zm {\mathbb{Z}}

\def\O{\mathcal{O}}

\newcommand{\cout}[1]{}

\newcommand{\x}{x}

\newcommand{\dprod}[2]{\left\langle{#1},{#2}\right\rangle}

\newcommand{\zbar}{\overline{z}}

\def \i {\boldsymbol\iota}

\newcommand{\Hksol} { {H_k^\text{sol}(M)} }

\newcommand{\ff}[1] {#1}

\title{On solenoidal-injective and injective ray transforms of tensor fields on surfaces}
\author{Venkateswaran P. Krishnan\thanks{Tata Institute of Fundamental Research, Centre For Applicable Mathematics, Bangalore, India. vkrishnan@math.tifrbng.res.in} \and Rohit Kumar Mishra\thanks{Department of Mathematics, University of California, Santa Cruz, CA 95064, USA. rokmishr@ucsc.edu} \and Fran\c cois Monard\thanks{Department of Mathematics, University of California, Santa Cruz, CA 95064, USA. fmonard@ucsc.edu}}

\begin{document}
\maketitle

\begin{abstract}
	We first give a constructive answer to the attenuated tensor tomography problem on simple surfaces. We then use this result to propose two approaches to produce vector-valued integral transforms which are fully injective over tensor fields. The first approach is by construction of appropriate weights which vary along the geodesic flow, generalizing the moment transforms. The second one is by changing the pairing with the tensor field to generate a collection of transverse ray transforms. 
\end{abstract}

\section{Introduction}\label{sec:introduction}

Let $(M,g)$ be a non-trapping Riemannian surface with strictly convex boundary $\partial M$. Denote its unit tangent bundle by 
\begin{align*}
    SM := \{(x,v)\in TM, \quad g_x(v,v) =1\},
\end{align*}
with inward and outward boundaries 
$$ \partial_{\pm} SM =  \{(x, v) \in SM : x \in \partial M \mbox{ and } \pm \langle v, \nu(x) \rangle \leq 0\}$$ 
where $\nu(x)$ is the outer unit normal to $\partial M$ at $x$. One may then define the geodesic flow $\varphi_t\colon SM\to SM$, with infinitesimal generator $X_{(x,v)}:= \frac{d}{dt}|_{t=0} \varphi_t(x,v)$. Given $F:SM\to \Cm$ and $w:SM\to \Cm$ a non-vanishing weight, the weighted ray transform $I_w:L^2(SM)\to L^2(\partial_+ SM)$ is defined by 
\begin{align*}
    I_{w} F(x,v) = \int_0^{\tau(x,v)} w(\varphi_t(x,v)) F (\varphi_{t}(x,v)) \ dt, \qquad (x,v)\in \partial_+ SM,
\end{align*}
and several problems of integral geometry arise from restricting such transforms to specific types of integrands $F$, for example functions on $M$, vector fields and tensor fields. In particular, for $f$ a symmetric $m$-tensor field, one may define the transform
\[ I_{w,m} f := I_w [\ell_m f], \qquad f\in C^\infty(S^m(T^* M)),  \]
%\begin{align*}
%    I_{w} \ell_m f(x,v) = \int_0^{\tau(x,v)} w(\varphi_t(x,v)) \ell_m f (\varphi_{t}(x,v)) \ dt, \qquad (x,v)\in \partial_+ SM,
%\end{align*}
where we identify a tensor field $f$ with its longitudinal restriction $\ell_m f$ to $SM$ defined by
\begin{align}
    \ell_m f (x,v) := f_x(v, \dots, v), \qquad (x,v)\in SM.
    \label{eq:lm}
\end{align}

Examples of such transforms in the literature arise commonly for $w \equiv 1$ (the geodesic X-ray transform), or when 
\[ w(x,v) = w_a(x,v) = \exp\left( -\int_0^{\tau(x,v)} a(\gamma_{x,v}(s))\ ds \right) \]
for some attenuation function $a(x) \in C^\infty(M)$. The associated transform is the attenuated X-ray transform, which we denote $I_a$ instead of $I_w$. With or without attenuation, it is well-known that the transform $I_{a,m}$ has a natural kernel, which increases with $m$, and the {\em tensor tomography problem} asks whether this natural kernel is the only obstruction to injectivity. \ff{Since a natural complement to this kernel is made of solenoidal tensor fields, the injectivity question is reformulated as ``s(olenoidal)-injectivity'', or injectivity over solenoidal tensor fields.} When $I_{a,m}$ is s-injective, further interest is given toward finding a reconstruction method for a representative of $f$ modulo the kernel of $I_{a,m}$ \ff{(in particular, the solenoidal representative may or may not be the most practical choice, as argued for instance in \cite{Monard2015a})}. Similar results and studies exist on higher-dimensional manifolds, closed manifolds\ff{, vector bundles} and for other types of flows, see the recent review article \cite{Ilmavirta2018}.

In the absence of attenuation, the answer has been shown to be positive, most recently on simple surfaces \cite{Paternain2011a}, and earlier on disks with a spherically symmetric metric satisfying Herglotz' non-trapping condition \cite{Sharafudtinov1997}, none of which generalizing the other. \ff{Explicit inversion approaches in the Euclidean case have been proposed in \cite{Kazantsev2004,Derevtsov3,Monard2015a}. The approach in \cite{Kazantsev2004} relies on the potential-solenoidal decomposition of tensors and SVD bases for solenoidal tensor fields; in \cite{Derevtsov3}, a characterization of solenoidal tensors in terms of higher powers of an operator ${\bf d}^\perp$ allows to set up an inversion procedure; finally in \cite{Monard2015a}, the third author relies on another decomposition of tensor fields which is generalized below in Theorem \ref{thm:decomposition}. Approaches were also proposed to reconstruct the singular support of vector fields \cite{Derevtsov2} and tensor fields \cite{Derevtsov4}.} It is conjectured that the answer to the tensor tomography problem is positive for any non-trapping surface with strictly convex boundary. In this direction, injectivity has recently been proved over piecewise constant functions in \cite{Ilmavirta2017}.

In the presence of attenuation, tensor tomography was solved for the cases $m=0,1$ in \cite{Salo2011}, with constructive approaches in \cite{Monard2015,Assylbekov2017}, \ff{some implemented in \cite{Monard2015}, and independent numerical methods in \cite{Derevtsov1}}; the case for general $m$ was studied in \cite{Monard2017a} in the Euclidean case. The latter provides a fully constructive answer, heavily relying on the ability to construct explicit invariant distributions with prescribed average on the fibers of $SM$. In the case of simple surfaces, new recent formulas were provided for functions and vector fields in \cite{Assylbekov2017}, following earlier works in \cite{Salo2011,Monard2015}. A first salient feature of our work is to propose a constructive solution to the attenuated tensor tomography problem on simple surfaces, see Theorem \ref{thm:ATT} below. This builds upon recent inversion formulas derived in \cite{Assylbekov2017} for sums of functions and one-forms. To generalize this to general $m$, we use that a tensor field of arbitrary order admits a gauge representative which differs from a [function, one-form] pair by higher-order, divergence-free, trace-free elements, see Theorem \ref{thm:decomposition} below. It is then enough to first reconstruct these, via appropriate pairing of the data with traces of special invariant distributions, see Section \ref{sec:ATTP}.

The second point of focus of this article is to study vector-valued ray transforms which are fully injective over tensor fields, some coined {\em moment ray transforms}, others {\em transverse ray transforms}. To the authors' knowledge, both were first introduced V.A. Sharafutdinov, the former in \cite{Sharafutdinov1986} and the latter in \cite[Chapter 5]{Sharafudtinov1997}, with applications to polarization tomography, see also \cite{Hammer2004,Palamodov2015,Derevtsov3}. We propose a generalization of both types of transforms to a Riemannian setting, and show that such transforms provide natural prototypes of injective ray transforms over tensor fields, constructively invertible in certain cases, even in the presence of certain weights, see Theorems \ref{thm:moments}, \ref{thm:transverse} and \ref{thm:inj} below. Prior results in this direction cover the case of vector fields \cite{Svetov2013} and we generalize them to tensor fields of arbitrary order. 

We now state the main results and provide an outline of the article below.

%%%%%%%%%%%%%%%%%%%%%%%%%%%%%%%%%%%%%%%%%%%%%%%%%%%%%%%%%%%%%%%%%%%%%%%%%%%%%%%%%%%%%%%% MAIN RESULTS
\section{Statement of the main results} \label{sec:main}

\subsection{Attenuated tensor tomography on simple surfaces} 

We first provide a positive answer to the attenuated tensor tomography problem on simple surfaces. Recall that a Riemannian surface $(M,g)$ is called {\em simple} if it is non-trapping, has strictly convex boundary, and has no conjugate points. Below, we will exploit Fourier analysis on the tangent circles, indexed by integer frequencies, and a function on $SM$ is said to have degree $m$ if it is supported on the harmonics $k$ such that $|k|\le m$. 

\begin{theorem}\label{thm:ATT}
    Let $(M,g)$ a simple Riemannian surface with boundary and $a\in C^\infty(M,\Cm)$. Suppose $f$ is a function of degree $m$ on $SM$ such that $I_a f = 0$. Then there exists $p$ of degree $m-1$ with $p|_{\partial SM} =0$ such that $f = Xp + ap$.
\end{theorem}

The interpretation in the language of tensor fields is as follows. We recall that for a tensor field $h$ of degree $m$, the function $\ell_m h$ on $SM$ is supported in the harmonics $-m, -m+2, \dots, m-2, m$. Since the presence of attenuation mixes even and odd Fourier modes, the result is most naturally stated in terms of pairs of tensor fields, in the sense that if $(f_m,f_{m-1})$ is a pair of two tensor fields of consecutive orders and $I_a(f_{m-1}+f_m) = 0$, then there exists a tensor field $p$ of order $m-1$ vanishing at $\partial M$ such that $f_m = d^s p$ and $f_{m-1} = ap$, with $d^s$ the symmetrized covariant derivative. 

Theorem \ref{thm:ATT} is based on the decomposition theorem below, which shows that a general $m$-tensor $f$ always differs from a ``potential tensor'' $Xp+ap$ by a unique representative over which the attenuated ray transform is injective and explicitly invertible. In the statement below, the space $W_0^{1,2}(M)$ is the standard Sobolev space, and the notation $\Hksol$ corresponds to square-integrable, trace-free, divergence-free $k$-tensors, see Sec. \ref{sec:prelims}. Below, by a function of ``degree $m$'' on $SM$, we mean a function in $\bigoplus_{k=0}^m \Omega_k$, see Sec. \ref{sec:prelims}.

\begin{theorem}\label{thm:decomposition}
    Let $(M,g)$ be a simply connected Riemannian surface and let $f\in L^2(S^m(T^*M))$. Then $\ell_m f$ decomposes uniquely as 
    \begin{align*}
	\ell_m f = (X+a)p + h,
    \end{align*}
    where $p$ is of degree $m-1$ with components in $W^{1,2}_0(M)$ and where 
    \begin{align}
	h = h_0 + X_\perp h_\perp + \sum_{k=1}^m h_k	
	\label{eq:hform}
    \end{align}
    with $h_0\in L^2(M)$, $h_\perp \in W^{1,2}_0(M)$ and for $k\ge 1$, $h_k \in \Hksol$. If $f\in C^\infty(S^m(T^* M))$, then $p$, $h_0$, $h_\perp$, $h_k$ are all smooth. 
\end{theorem}

In light of Theorem \ref{thm:decomposition}, the proof of Theorem \ref{thm:ATT} consists in proving that the transform $I_a$ is injective over integrands of the form \eqref{eq:hform}. 

Such reconstruction approaches provide the building blocks for invertibility and inversion of the ray transforms considered below.

\subsection{Moment transforms} 

As mentioned above for $m\ge 1$, the problem of reconstructing $f\in S^m$ from $If$ is non-injective. A first approach to recover injectivity is to consider ray transforms involving {\em higher moments} along each geodesic. This was previously tackled by Sharafutdinov in the Euclidean case in \cite{Sharafudtinov1994}, see also \cite{Abhishek2017}. In particular, if $\tau(\x,v)$ denotes the first time at which the geodesic emanating from $(x,v)$ hits the boundary, for any $k\ge 0$, we define the $k$-th moment ray transform of $f$ as
\begin{align*}
  I_{[k]} F(\x,v) = \int_0^{\tau(\x,v)} (\tau(\x,v) - t)^k F(\varphi_t(\x,v))\ dt = \int_0^{\tau(\x,v)} (\tau^k F) (\varphi_t(\x,v))\ dt, \qquad (\x,v)\in \partial_+ SM. 
\end{align*}
It is proved in \cite{Sharafudtinov1994} in the Euclidean case that the moments of orders $0\le k\le m$ determine an $m$-tensor field, i.e. the case where $F = \ell_m f$ for some $f\in C^\infty(S^m (T^*M))$. A reconstruction algorithm and Reshetnyak stability estimates for moment ray transforms has been derived in a recent work  as well \cite{KMSS}. We show that unique determination of a tensor field from its moment  ray transform generalizes to Riemannian settings, and in the presence of attenuation coefficients. % upon changing the moments adequately. 

%\begin{theorem} A symmetric $m$-tensor $f$ is uniquely characterized by the knowledge of the transforms $I_{[k]} f$ for all $0\le k\le m$, with an explicit reconstruction procedure.   
%  \label{thm:moments_inversion}
%\end{theorem}

%\begin{proof}
%  We prove this by induction. The case $m=0$ is just the injectivity of $I_0$. Now suppose the statement holds for some $m$, let $f$ and $m+1$ tensor and consider the reconstruction of $f$ from $I_{[k]}f$ for $0\le k\le m+1$. Following Theorem \ref{thm:decomposition}, we write $f = Xv + g$ with $v\in S^m$ such that $v|_{\partial M} =0$. Then $I_{[0]}f = I_{[0]}g$, out of which one can reconstruct $g$. Then for each $0\le k\le m$, 
%  \begin{align*}
%    I_{[k+1]} f = I_{[k+1]} g + I (\tau^{k+1} Xv) = I_{[k+1]}g + \underbrace{I(X(\tau^{k+1} v))}_{=0} - I(X(\tau^{k+1})v) = I_{[k+1]} g + (k+1) I_{[k]} v. 
%  \end{align*}
%  In other words, the data 
%  \begin{align*}
%    I_{[k]} v = \frac{1}{k+1} \left( I_{[k+1]}f - I_{[k+1]}g \right), \qquad 0\le k\le m,
%  \end{align*}
%  is known from data, out of which we can reconstruct $v$ using the induction property. Theorem \ref{thm:moments_inversion} is proved.
%\end{proof}

\begin{theorem} \label{thm:moments} Let $(M,g)$ a non-trapping surface, $a\in C^\infty(M)$ and suppose that $I_{a,k}$ is s-injective for all $0\le k\le m$. Then for any $m$-tensor field $f\in C^\infty(S^m (T^*M))$, the collection of moment ray transforms 
    \begin{align*}
	I_{a,m} [\tau^k \ell_m f], \qquad 0\le k\le m	
    \end{align*}
    determines $f$ uniquely and explicitly.
\end{theorem}

\subsection{Transverse transforms} 

The second class of vector-valued transforms is a collection of {\em transverse ray transforms}, previously appearing in \cite{Sharafudtinov1994,Hammer2004,Palamodov2015,Svetov2013} as mentioned in the Introduction. In what follows, we assume $(M,g)$ oriented\footnote{This is always true since our resting assumption is that $(M,g)$ be nontrapping.}, giving rise to an operator $v\mapsto v^\perp$ of direct rotation by $\pi/2$. Similar to the operator $\ell_m$ given in \eqref{eq:lm}, we define the family of operators $\ell_{m,k}\colon C^\infty (S^m(T^*M))\to C^\infty (SM)$ by
\begin{align}
    \ell_{m,k} f (x,v) = f_x (v^{m-k}\otimes (v^\perp)^k), \qquad 0\le k\le m \qquad (\ell_m \equiv \ell_{m,0}).
    \label{eq:lmk}
\end{align}

Each of these identifications can then define a tranverse ray transform by means of integration, and we provide conditions under which such a collection gives an injective ray transform over tensor fields. The proof relies on an algebraic reduction to the injectivity of scalar transforms defined over symmetric differentials, previously studied in \cite{Monard2013a}. In the non-trapping case where global isothermal coordinates $(x,\theta)$ exist on $M$, studying the ray transform $I_w$ over $k$-differentials is equivalent to the transform $L^2(M)\ni h\mapsto I_w [h e^{ik\theta}]$.

\begin{theorem}\label{thm:transverse}
    Fix $m$ any natural integer. Suppose $(M,g)$ is a Riemannian surface with boundary, and let $w:SM\to \Cm$ a weight such that for any $k$ with $m-k$ even and $|k|\le m$, the scalar transform $L^2(M)\ni h\mapsto I_w [h e^{ik\theta}]$ is injective. Then for any $f\in C^\infty(S^m(T^*M))$, the collection of transverse ray transforms 
    \begin{align}
	I_{w,m} [\ell_{m,k}f], \qquad 0\le k\le m,
	\label{eq:transverse}
    \end{align}
    determines $f$ uniquely and explicitly.
    %the transform
    %\begin{align*}
	%\I_{w,m}:L^2(S^m(T^*M))\to (L^2(\partial_+ SM))^{m+1},
    %\end{align*}
    %defined by 
    %\begin{align*}
	%[I_{w,m} f]_k (x,v) = \int_{0}^{\tau(x,v)} w(\varphi_t(x,v)) \ell_{m,k} f(\varphi_t(x,v))\ dt, \quad (x,v)\in \partial_+ SM, \qquad 0\le k\le m,
    %\end{align*}
    %is injective.
\end{theorem}

\paragraph{Summary of injective settings.} So far the injectivity results we have stated rely on the injectivity of other, simpler transforms. We now summarize on what surfaces Theorems \ref{thm:moments} and \ref{thm:transverse} translate into injective transforms over tensor fields.  

\begin{theorem} \label{thm:inj}
    Let $(M,g)$ a simply connected surface with boundary and suppose $w \equiv w_a$ for some attenuation $a\in C^\infty(M)$.  Then in either case below, for any $m\ge 0$, an $m$ tensor $f\in C^\infty(S^m(T^* M))$ is uniquely characterized by the collection of moment transforms $\{I_{a,m} [\tau^k \ell_{m}f]\}_{m=0}^k$, or the collection of transverse ray transforms $\{I_{a,m} [\ell_{m,k}f]\}_{m=0}^k$:
    %the transform $\I_m:L^2(S^m (T^*M))\to L^2(\partial_+ SM, \Cm^{m+1})$ is injective: 

    $(i)$ $(M,g)$ is simple.

    $(ii)$ $(M,g)$ is a disk endowed with a radial metric satisfying Herglotz' non-trapping condition, and $a = 0$. 
\end{theorem}

\begin{proof}
    Since it was proved in \cite{Monard2013a} on a simple surface that for any $k\in \Zm$, the transform $L^2(M)\ni f\mapsto I[e^{ik\theta} f]$ is injective, $(i)$ follows directly from Theorem \ref{thm:transverse} for the transverse transforms. The case of moment transforms follows from Theorems \ref{thm:ATT} and \ref{thm:moments}.

    To prove $(ii)$, it was proved in \cite{Sharafudtinov1997} in this context that the geodesic X-ray transform is s-injective over $m$-tensors for any $m\ge 0$. Therefore the result for moment transforms follows immediately from Theorem \ref{thm:moments} while for transverse ray transforms, the result follows by using Lemma \ref{lem:omegak}.
\end{proof}

\paragraph{Outline.} The remainder of the paper is organized as follows. We recall some geometric preliminaries in Section \ref{sec:prelims}. Section \ref{sec:ATTP} deals with attenuated tensor tomography and provides a proof of Theorems \ref{thm:decomposition} and \ref{thm:ATT}. Section \ref{sec:moments} covers the proof of Theorem \ref{thm:moments} on moment transforms and Section \ref{sec:transverse} covers the proof of Theorem \ref{sec:transverse} on transverse ray transforms.

%%%%%%%%%%%%%%%%%%%%%%%%%%%%%%%%%%%%%%%%%%%%%%% PRELIMINARIES

\section{Preliminaries} \label{sec:prelims}

First note that the non-trapping assumption implies that $M$ is simply connected, hence oriented. In particular, there exists a circle action on the fibres of the unit tangent bundle $SM$ which is generated by the vertical vector field $V$, and we will use the well-known canonical framing of $T(SM)$ given by $\{X,V, X_\perp := [X,V]\}$, where $[\cdot, \cdot]$ denotes the Lie bracket of vector fields. $SM$ is equipped with the Sasaki metric making $(X,X_\perp,V)$ orthonormal and the $L^2(SM)$ inner product 
\begin{align*}
    \langle u, v \rangle_{SM} = \int_{SM} u\ \overline{v}\ d\Sigma^3
\end{align*}
is defined with respect to its corresponding volume form. For $u \in C^\infty(SM)$ vanishing on $\partial_- SM$, we have the following integration by parts formula
\begin{align}
    \int_{SM} Xu\ d\Sigma^3 = - \int_{\partial_+ SM} u(x,v) \mu(x,v) d\Sigma^2, \qquad \mu(x,v) := |g_x(\nu(x), v)|.
    \label{eq:IBP}
\end{align}

%We will define the $L^2(SM)$ inner product of two functions $u, v:SM  \rightarrow \mathbb{C}$ as the following:
%\begin{align*}
%\langle u, v \rangle_{L^2(SM)} = \int_{SM} u  \bar{v}\  d \Sigma
%\end{align*}
%where $d\Sigma$ denotes the Liouville measure of the metric $g$ on $SM$.

We can decompose $L^2(SM)$ orthogonally by diagonalizing the vertical Laplacian $-V^2$, as the following direct sum:
\begin{align*}
L^2(SM) = \bigoplus_{k \ge 0} H_k, \qquad H_k := \ker (-V^2 - k^2 Id)\cap L^2(SM).
\end{align*}
In what follows, we may also denote $\Omega_k :=  H_k \cap C^\infty(SM)$. Following notation in \cite{Paternain2015}, for each $k>0$, the space $H_k$ decomposes into $H_k = E_k \oplus E_{-k}$, (resp. $\Omega_k = \Lambda_k \oplus \Lambda_{-k}$), where $E_{\pm k}$ (resp. $\Lambda_{\pm k}$) corresponds to $L^2$ (resp. smooth) sections of $\ker (-iV\pm k Id)$. 

We will decompose an element in $L^2(SM)$ (or $C^\infty(SM)$) using this orthogonal decomposition, as follows
\begin{align*}
    u  = \sum_{k = 0}^\infty  u_k, \qquad u_k \in H_k.
\end{align*}
In isothermal coordinates $(x, \theta)$, the component $u_k$ takes the form
\[ u_k(x, \theta) =u_{k,+}(x)e^{ik\theta}+ u_{k,-}(x)e^{-ik\theta}, \qquad u_{k, \pm} \in L^2(M). \]
In these coordinates, an element $u\in H_k$ can be written as $u(x,\theta) = \tilde u(x) e^{ik\theta}$ with $\tilde u\in L^2(M)$, so we will somewhat abuse notation by defining
\begin{align*}
    e^{ik\theta} F(M) = \{u\in H_k, \quad u = e^{ik\theta} \tilde u(x), \quad \tilde u \in F(M)\},
\end{align*}
for $F$ some function space (e.g., $W^{1,2}_0(M)$). Though the notation suggests isothermal coordinates, these spaces do not depend on the choice of smooth abelian differential (e.g., $e^{i\theta}$) whose powers are used to factor out the fiber dependence.  

In what follows, we will also use the splitting $X = \eta_+ + \eta_-$ first introduced by Guillemin and Kazhdan in \cite{Guillemin1980}, where
\begin{align*}
\eta_{\pm} = \frac{1}{2}\left(X\pm i X_\perp\right),
\end{align*}
with the property that $\eta_\pm (\Omega_k) \subset \Omega_{k\pm 1}$ for all $k\in \mathbb{Z}$. In what follows, we also denote
\begin{align*}
   \ker^k\  \eta_{\pm} := \Omega_k \cap \mbox{ ker } \eta_{\pm}, \qquad k\in \mathbb{Z}, 
\end{align*}
as well as 
\begin{align}\label{eq: definition kernel of eta}
    L^2(\ker^{\mp k} \eta_{\pm}) := \{ f \in L^2(SM): f_p =0 \mbox{ for } p \neq \mp k; \quad \eta_{\pm} f  =0\}.
\end{align}
As explained in \cite[Sec. 7.1]{Assylbekov2017}, such spaces are closed subspaces of $L^2(SM,\Cm)$ (or $H_k$), and as such are Hilbert spaces themselves, admitting complete orthonormal sets, which we denote $\{ \phi^{\mp k, j}\}_{j=0}^{\infty}$.

Of special interest will be, for $k \ge 1$, the spaces $\ker^k \eta_- \oplus \ker^{-k} \eta_+$. Via the identification $\ell_k$ in \eqref{eq:lm} for $k\ge 2$, such spaces correspond to trace-free, divergence-free tensors of order $k$.
The $L^2$ version will be denoted
\begin{align*}
    \Hksol := L^2(\ker^k \eta_-) \oplus L^2(\ker^{-k} \eta_+), \qquad k\ge 1.
\end{align*}

%%%%%%%%%%%%%%%%%%%%%%%%%%%%%%%%%%%%%%%%%%%%%%%% LONGITUNDINAL

\section{Scalar attenuated transforms - Proof of Theorems \ref{thm:decomposition} and \ref{thm:ATT}} \label{sec:ATTP}

We first prove the decomposition Theorem \ref{thm:decomposition}, whose proof relies on the following lemma. 

\begin{lemma}\label{lem:decomp}
    Let $f_k \in H_k$ for $k \geq 2$. Then there exist $h_k \in \Hksol$, $g_{k-1} \in H_{k-1}$ with $g_{k-1, \pm} \in W^{1,2}_0(M)$, and $w_{k-2} \in H_{k-2}$, such that 
    \begin{align}\label{eq:decomposition of f_k}
	f_k = X g_{k-1}+w_{k-2} +h_k.
    \end{align} 
\end{lemma}

To prove the lemma, recall the following two elliptic decompositions: for any $k \in \Zm$, any $f_k \in E_k$ can be uniquely written in the following two ways: 
\begin{align}
    \begin{split}
	f_k &= \eta_+ g_{k-1} + h_k, \qquad g_{k-1} \in e^{i(k-1)\theta} W^{1,2}_0(M), \qquad h_k \in L^2(\ker^k \eta_-), \\
	f_k &= \eta_- g'_{k+1} + h'_k, \qquad g'_{k+1} \in e^{i(k+1)\theta} W^{1,2}_0(M), \qquad h'_k \in L^2(\ker^k \eta_+).
    \end{split}
    \label{eq:elliptic}
\end{align}

\begin{proof}[Proof of Lemma \ref{lem:decomp}] Let $f_k  = f_{k,+} + f_{k,-}$. Using \eqref{eq:elliptic}, we can write
    \begin{align*}
	f_{k,+} &= \eta_+ g_{k-1, +} + h_{k, +}, \quad h_{k, +}\in L^2(\ker^k \eta_-), \quad g_{k-1,+} \in W^{1,2}_0(M), \\
	\text{and} \qquad f_{k,-} &= \eta_- g_{k-1, -} + h_{k, -}, \quad h_{k, -}\in L^2(\ker^{-k} \eta_+), \quad g_{k-1,-} \in W^{1,2}_0(M).
    \end{align*}
    Define $h_k := h_{k,+}+ h_{k,-}$ and $g_{k-1} :=g_{k-1,+} + g_{k-1,-}$. Using these we rewrite $f_k$ as follows:
    \begin{align*}
	f_k &= \eta_+ g_{k-1, +} +\eta_- g_{k-1, -} + h_k\\
	&=  X g_{k-1} - \eta_- g_{k-1, +} -\eta_+ g_{k-1, -} + h_k,
    \end{align*}
    and the proof follows upon setting $w_{k-2} := - \eta_- g_{k-1, +} -\eta_+ g_{k-1, -}$.
\end{proof}

\begin{proof}[Proof of Theorem \ref{thm:decomposition}.]
    The proof uses induction on $m$ and Lemma \ref{lem:decomp}. The case $m=0$ follows trivially by taking $h_0 = f_0$. For the case $m=1$, we start by writing as $f = f_0 + f_1$ and decompose $f_1$ according to \eqref{eq:elliptic}:
    \begin{align*}
	f_{1,+} &=  \eta_+ g_{0,+} + h_{1,+}, \quad h_{1, +}\in L^2(\ker^1 \eta_-), \quad g_{0,+} \in W^{1,2}_0(M), \\
	\text{and} \qquad f_{1,-} &=  \eta_- g_{0,-} + h_{1,-}, \quad h_{1, -}\in L^2(\ker^{-1} \eta_+), \quad g_{0,-} \in W^{1,2}_0(M).    
    \end{align*}
    Using the identities $\eta_{\pm} = (X\pm i X_{\perp})/2$, we see 
    \begin{align*}
	\eta_+ g_{0,+} + \eta_- g_{0, -} = X\underbrace{ \left(\frac{g_{0,+} +g_{0,-}}{2}\right)}_{p} + X_{\perp}\underbrace{\left(i \frac{g_{0,+} -g_{0,-}}{2}\right)}_{h_\perp}.
    \end{align*}
    From this we rewrite $f$ as 
    \begin{align*}
	f &= f_0 + X p + X_\perp h_\perp  + h_1\\
	&= (X +a)p +\underbrace{(f_0 - a p)}_{h_0}+ X_\perp h_\perp  + h_1.
    \end{align*}
    %Putting this in the transport equation $ (X+a) u = -f$, we have
    %\begin{align*}
    %(X+a) u &= -(f_0 + X h_p + X_\perp h_\perp  + h_1)\\
    %(X+a)(u+h_p)&= -(f_0 - a h_p)-X_\perp h_\perp - h_1
    %\end{align*}
    with $h_0\in L^2(M)$, $h_\perp \in W^{1,2}_0(M)$ and $h_1 \in H_1^\text{sol}(M)$. Hence the result is true for $m=1$ also.
    \par We now show the induction step $(k \implies k+1)$. Let $f\in L^2(S^{k+1}(T^*M))$ and write $ f  = f_{\leq k} + f_{k+1}$ with $f_{\leq k} \in L^2(S^{k}(T^*M))$. By Lemma \ref{lem:decomp}, $f_{k+1}$ decomposes into
    \begin{align*}
	f_{k+1} = X g_{k}+w_{k-1} +h_{k+1}
    \end{align*}
    where $h_{k+1} \in H_{k+1}^{\text{sol}}(M)$, $g_{k} \in e^{ik\theta}W^{1,2}_0(M) \oplus e^{-ik\theta} W^{1,2}_0(M)$ and $w_{k-1} \in H_{k-1}$. Using this we rewrite $f$ as
    \begin{align*}
	f &= f_{\leq k} + X g_{k}+w_{k-1} +h_{k+1}\\
	&= (X+a)g_{k} +\underbrace{ w_{k-1} - a g_{k} +f_{\leq k}}_{\tilde{f}} +h_{k+1}.
	%&= (X+a)(g_{k} +p) +h_0  + X_\perp h_\perp +w_{k-1} - a g_{k} + \sum_{j=1}^{k+1} h_j
    \end{align*} 
    Then we can use our induction hypothesis to decompose $\tilde{f}\in L^2(S^{k}(T^*M))$  as 
    \begin{align*}
	\tilde{f} = (X+a)\tilde{p} + h_0 + X_\perp h_\perp + \sum_{j=1}^{k} h_j.
    \end{align*}  
    Finally we put it back to $f$ and get the following:
    \begin{align*}
	f &= (X+a)g_{k} +(X+a)\tilde{p} + h_0 + X_\perp h_\perp + \sum_{j=1}^{k} h_j +h_{k+1}\\
	&= (X+a)p + h_0 + X_\perp h_\perp + \sum_{j=1}^{k+1} h_j, \quad \mbox{ where } p := \tilde{p}+ g_{k}.
    \end{align*}
    Theorem \ref{thm:decomposition} is proved.
\end{proof}

We move on to the proof of Theorem \ref{thm:ATT}. The following Lemma was first proved in \cite[Lemma 7.2]{Assylbekov2017} for the case $k=1$, and we now generalize it to arbitrary $k$.  %The following is a straighforward generalization to arbitrary $k$. 
Here and below, we denote by $\mathcal{O}_{\geq (k+1)}$ an element of $\bigoplus_{p\ge k+1} \Lambda_p$, similarly for $\mathcal{O}_{\le (k+1)}$. 

\begin{lemma}\label{lem:phi}
    \begin{itemize}
	\item[(a)] For every $\phi \in$ ker$^k \eta_-$, there exists $\psi =  \phi + \mathcal{O}_{\geq (k+1)}$ solution of $X\psi - \bar{a}\psi =0$.
	\item[(b)] For every $\phi \in$ ker$^{-k} \eta_+$, there exists $\psi =  \phi + \mathcal{O}_{\leq -(k+1)}$ solution of $X\psi - \bar{a}\psi =0$.
    \end{itemize}
\end{lemma}

\begin{proof}
    Let $ w $ and $\tilde{w}$ be solutions of $ Xw = X\tilde{w} = \bar{a}$ obtained from \cite[Proposition 4.1]{Salo2011} where $w$ is holomorphic and $\tilde{w}$ is antiholomorphic. Using these solutions we get holomorphic solution $e^w = 1 + \mathcal{O}_{\geq 1}$ of  $(X- \bar{a}) w =0$ and antihomorphic solution $e^{\tilde{w}} = 1 + \mathcal{O}_{\leq 1}$ of  $(X- \bar{a}) \tilde{w} =0$.
    \par If $\phi \in \ker^k \eta_-$ then there exists a smooth solution $ v $ from \cite[Lemma 5.6]{Paternain2013a} of $Xv =0$ such that $ v_k = \phi$. Now $\eta_- v_k = \eta_- \phi =0$ implies that $\tilde{v} = \sum_{j\geq k} v_j$ is also  satisfies $X \tilde{v} =0$. Finally define $\psi  = e^w \tilde{v}$ to complete the proof of $(a)$. 
    \par If $\phi \in \ker^{-k} \eta_+$ then there exists a smooth solution $ v $ from \cite[Lemma 5.6]{Paternain2013a} of $Xv =0$ such that $ v_{-k} = \phi$. Now $\eta_+ v_{-k} = \eta_+ \phi =0$ implies that $\tilde{v} = \sum_{j\leq -k} v_j$ is also  satisfies $X \tilde{v} =0$. Finally define $\psi  = e^{\tilde{w}} \tilde{v}$ to complete the proof of $(b)$.
\end{proof}

We are now ready to prove Theorem \ref{thm:ATT}. 

\begin{proof}[Proof of Theorem \ref{thm:ATT}] In light of the decomposition Theorem \ref{thm:decomposition}, it is enough to show how to reconstruct $h$ from $I_a f = I_a h$.  
    
    We are going to show that $h_k$ for $2 \leq k \leq m$ can be reconstructed. As a first step we will reconstruct $h_m$. Let us write $$h_m = h_{m,+} + h_{m,-}$$ 
    where $h_{m,+} \in L^2(\mbox{ker}^m \eta_{-})$ and $h_{m,-} \in L^2(\mbox{ker}^{-m} \eta_{+})$. We will prove how to reconstruct $h_{m,+} $ from the knowledge of $I_a h$ and the proof for the reconstruction of $h_{m,-}$ is similar.
    Since $h_{m,+} \in L^2(\ker^m \eta_{-})$, it can be expressed as follows
    \begin{align*}
	h_{m,+} =  \sum_{j=0}^\infty \langle h_{m,+}, \phi^{+ m, j} \rangle_{SM}\ \phi^{+ m, j}.
    \end{align*}
    We use Lemma \ref{lem:phi}.$(a)$ to construct $\psi^{+m, j}$ corresponding to each  $\phi^{+ m, j}$. Take the inner product of the equation 
    $$ (X + a) u =- (h_0 + X_\perp h_s + \sum_{j=1}^m h_j) $$ 
    with  $\psi^{+m, j}$ to get 
    \begin{align}
	\langle (X + a) u, \psi^{+m, j}\rangle_{SM} &=  - \langle (h_0 + X_\perp h_s + \sum_{j=1}^m h_j), \psi^{+m, j}\rangle_{SM}.
	\label{eq:tmp}
    \end{align}
    Upon using integrating by parts \eqref{eq:IBP} and the fact that $u|_{\partial_+ SM} = I_a h$, the left-hand side of \eqref{eq:tmp} becomes 
    \begin{align*}
	\int_{\partial_+ SM} I_a h\  \overline{\psi^{+m, j}_{\partial_+ SM}}\ \mu\ d\Sigma^2 - \langle u, (-X + \bar{a})\psi^{+m, j}\rangle_{SM} = \int_{\partial_+ SM} I_a h\  \overline{\psi^{+m, j}_{\partial_+ SM}}\ \mu\ d\Sigma^2 
    \end{align*}
    while by consideration of harmonic content, the right-hand side of \eqref{eq:tmp} reduces to $-\langle h_{m, +} , \phi^{+ m, j}\rangle_{SM}$. This is because the first term in the inner product is $\O_{\le (m)}$ while the second is $\O_{\ge (m)}$. We can then derive that $h_{m,+}$ can be reconstructed via
    \begin{align*}
	h_{m,+} =  \sum_{j=0}^\infty \left( \int_{\partial_+ SM} I_a h\  \overline{\psi^{+m, j}_{\partial_+ SM}}\ \mu\ d\Sigma^2 \right)\ \phi^{+ m, j}.	
    \end{align*}
    %\begin{align*}
	%\langle I_a h , \psi^{+m, j}_{\partial_+ SM}\rangle - \langle u, (-X + \bar{a})\psi^{+m, j}\rangle &= -\langle h_{m, +} , \phi^{+ m., j}\rangle\\
	%\Rightarrow \quad \quad \quad \quad \quad \quad \quad \quad \quad \quad \quad \quad \quad  \langle I_a h , \psi^{+m, j}_{\partial_+ SM}\rangle &= -\langle h_{m, +} , \phi^{+ m., j}\rangle.
    %\end{align*}
    Using similar analysis we can reconstruct $h_{m,-}$ from the knowledge of $I_a h$ and hence $h_m$ is reconstructed. Using the knowledge of $h_m$ we also know $I_a (h - h_m)$. If we denote $\tilde{h}  = h - h_m$  then again we can use the same procedure to compute $\tilde{h}_{m-1} =h_{m-1}$ from the knowledge of $I_a(\tilde{h})$. Repeating this process $(m-4)$- more times we can reconstruct $h_{m}$, $h_{m-1}$ down to $h_2$.
    
    Since $h_k$ for $2\le k\le m$ have been reconstructed, it remains to reconstruct $h_0, h_\perp, h_{-1}, h_{1}$ from $I_a [h_0 + X_\perp h_\perp + h_{-1} + h_1]$, and this is done explicitly in \cite[Sec. 7]{Assylbekov2017}. Theorem \ref{thm:ATT} is proved. 
\end{proof}

As suggested by an anonymous referee, a second proof of Theorem \ref{thm:ATT} can be written, making direct use of a holomorphic integrating factor for the attenuation function $a$. While not amenable to inversion techniques, the proof is direct and we record it here. 

\begin{proof}[Second proof of Theorem \ref{thm:ATT}] Suppose $I_a f =0$. Then the function $u$ solving $Xu + au = -f$ on $SM$ with boundary condition $u|_{\partial_- SM}=0$ also vanishes on $\partial_+ SM$, and the theorem is proved if we can show that $u_k = 0$ for all $k\ge m$. With $w$ a fiberwise holomorphic, smooth function such that $Xw = -a$, the function $e^{-w}u$ satisfies $X(e^{-w} u) = - e^{-w}f$ and vanishes on $\partial SM$. Since $(e^{-w}f)_{k,-} = 0$ for all $k\ge m$, this implies, by e.g. \cite[Prop. 4.2]{Paternain2011a}, that $(e^{-w}u)_{k,-} = 0$ for all $k\ge m-1$. In particular, since $e^{w}$ is fiberwise holomorphic, multiplying $e^{-w}u$ by $e^w$ will preserve this property, and we then obtain that $u_{k,-} = 0$ for all $k\ge m-1$. Using an antiholomorphic integrating factor together with \cite[Prop. 4.1]{Paternain2011a} will also yield $u_{k,+} = 0$ for all $k\ge m-1$. Hence the result.    
\end{proof}

%%%%%%%%%%%%%%%%%%%%%%%%%%%%%%%%%%%%%%%%%%%%%%%%%%%%%%%%% MOMENTS

\section{Moment transforms - proof of Theorem \ref{thm:moments}} \label{sec:moments}

\begin{proof}[Proof of Theorem \ref{thm:moments}] We prove this by induction. The case $m=0$ is just the injectivity of $I_{a,0}$. Now suppose the statement holds for some $m$, let $f$ an $m+1$ tensor and consider the reconstruction of $f$ from $I_{a,m}[\tau^k f]$ for $0\le k\le m+1$. Following Theorem \ref{thm:decomposition}, we write $f = (X+a)v + g$ with $v$ of degree $m$ such that $v|_{\partial SM} =0$. Then $I_{a,m} f = I_{a,m} g$, and one may reconstruct $g$ from $I_{a,m}g$ by virtue of Theorem \ref{thm:ATT}. Then for each $0\le k\le m$, noticing the identity
    \begin{align*}
	\tau^{k+1} f = \tau^{k+1} (X+a)v + \tau^{k+1} g = (X+a) (\tau^{k+1} v) + (k+1) \tau^k v + \tau^{k+1} g
    \end{align*}
    we arrive at the relation
    \begin{align*}
	I_{a,m} [\tau^{k+1} f] = \underbrace{I_{a} [(X+a) (\tau^{k+1} v)]}_{=0} + (k+1) I_{a,m-1} [\tau^k v] + I_{a,m} [\tau^{k+1} g].
  \end{align*}
  In other words, the data 
  \begin{align*}
      I_{a,m-1} [\tau^k v] = \frac{1}{k+1} \left( I_{a,m} [\tau^{k+1} f] - I_{a,m} [\tau^{k+1} g] \right), \qquad 0\le k\le m,
  \end{align*}
  is known from the initial data, out of which we can reconstruct $v$ using the induction property. Theorem \ref{thm:moments} is proved.
\end{proof}

\begin{remark} As building block to the proof above, one must be able to find a gauge representative and reconstruct it from the usual longitudinal X-ray transform. We do this through Theorems \ref{thm:decomposition} and \ref{thm:ATT}. Another decomposition from the one in Theorem \ref{thm:decomposition} is the solenoidal-potential one, where the representative to be reconstructed is a solenoidal tensor field. If an efficient reconstruction procedure can be derived for the solenoidal representative, then this provides another approach to prove Theorem \ref{thm:moments} and reconstruct tensor fields from their moment transforms. 
\end{remark}

%%%%%%%%%%%%%%%%%%%%%%%%%%%%%%%%%%%%%%%%%%%%%%%%%%%%% TRANSVERSE

\section{Transverse ray transforms - proof of Theorem \ref{thm:transverse}} \label{sec:transverse}

\subsection{On the injectivity of the transforms $I_{w} [f(x) e^{ik\theta}]$}

The injectivity of the transforms $I_{w} [f(x) e^{ik\theta}]$ was proved in \cite[Theorem 3.2(i)]{Monard2013a} for simple surfaces whenever the longitudinal transform is $s$-injective over $m$-tensors for some $m\ge k$ with $m-k$ even. This fact alone only relies on topological restrictions on $M$, as is explained below. The main tool is the splitting $X = \eta_+ + \eta_-$ into the two elliptic operators $\eta_{\pm} := X\pm i X_\perp$ such that, over any space $\Omega_k$, the problem
\begin{align*}
    \eta_\pm u = 0 \qquad (SM), \qquad u|_{\partial SM} = 0,
\end{align*}
only admits the trivial solution $u\equiv 0$. We then state the following result, whose proof is identical to \cite[Theorem 3.2(i)]{Monard2013a}

\begin{lemma} \label{lem:omegak} 
    Let $(M,g)$ simply connected, and suppose the geodesic ray transform is s-injective over $m$-tensors. Then for any $k\le m$ with $m-k$ even, the transform $L^2(M) \ni f \mapsto I[e^{ik\theta}f]$ is injective.     
\end{lemma}

%\begin{proof} Let $f$ such that $I[e^{ik\theta}f] = 0$ (suppose $k\ge 0$ WLOG). $e^{ik\theta}f$ can be regarded as the restriction to $SM$ of the $m$-tensor field $e^{m\lambda} f(x) (dz^{m-k}\otimes d\zbar^{k})$, whose ray transform vanishes. By solenoidal injectivity of the ray transform, there exists $u = u_{m-1} + u_{m-3} + \dots + u_{-m+1}$ of degree $m-1$ such that $Xu = -f$. Projecting onto the harmonics $p \in \{-m-2, m, \dots, m, m+2\}\backslash \{k\}$, and using the facts about $\eta_\pm$, we obtain that 
%    \begin{align*}
%	u_{-m+1} = u_{-m+3} = \dots = u_{k-1} &= 0 \\
%	u_{m-1} = u_{m-3} = \dots = u_{k+1} &= 0.
%    \end{align*}
%    Finally, the projection onto the harmonic $k$ gives
%    \begin{align*}
%	-f = \eta_+ u_{k-1} + \eta_- u_{k+1} = 0,
%    \end{align*}
%    hence the proof.
%\end{proof}

%\medskip
%\hrule
%\medskip

We now generalize the lemma above to a version for the attenuated ray transform. Since attenuation mixes odd and even Fourier modes, the condition ``$m-k$ even'' in the above statement is no longer in order, and the statement becomes: 

\begin{lemma} \label{lem:omegakatt} 
    Let $(M,g)$ simply connected, and suppose the transform $I_a$ is s-injective over $m$-tensors. Then for any $k$ with $|k|\le m$, the transform $L^2(M) \ni f \mapsto I_a[e^{ik\theta}f]$ is injective. 
\end{lemma}

\begin{proof}
 %   \F{do it}
    Let $f$ be such that $I_a[e^{ik\theta}f] = 0$ also assume $k\ge 0$ without loss of generality. We can think \tred{of} $e^{ik\theta}f$ as the restriction to $SM$ of the $m$-tensor field $e^{m\lambda} f(x) (dz^{m-k}\otimes d\zbar^{k})$, whose attenuated ray transform vanishes. Then $s$-injectivity of $I_a$ implies the existence of an $(m-1)$-tensor $v = \sum_{j=-(m-1)}^{m-1} v_j$ such that $(X+a)v =-f$ and $v_{j}|_{\partial M} =0$ for $j$ such that $|j|\le m-1$. 
Now for any  $ -m \leq j \leq m$ and $j \neq k$, we have 
    \begin{align*}
    \{(X+a)v\}_j = 0.
    \end{align*}
More explicitly we will get the following:
\begin{align*}
\eta_- v_{-(m-1)} &=0,\\
\eta_- v_{-(m-2)} + a v_{-(m-1)} &=0\\
\eta_+ v_{-(m-1)}+\eta_- v_{-(m-3)} + a v_{-(m-2)} &= 0,\\
& \ \ \vdots\\
\eta_+ v_{k-1}+\eta_- v_{k+1} + a v_k &= -f \\
& \ \ \vdots\\
\eta_+ v_{(m-2)} + a v_{(m-1)} &=0\\
\eta_+ v_{m-1} &=0.
\end{align*}
Augmented with homogeneous boundary conditions for each $v_k$, the first equation implies $v_{-(m-1)} = 0$, then the second $v_{-(m-2)} = 0$, down to the equation $\eta_+ v_{k-2} + \eta_- v_k + a v_{k-1} = 0$ which implies $v_k = 0$ (assuming $v_{k-1}= v_{k-2} = 0$ at this point). Similarly, by backward substitution we also have that $v_{m-1} = v_{m-2} = \dots = v_{k+1} = 0$. This implies $v\equiv 0$ and therefore $f = 0$. 
%We can solve these equation together with boundary conditions of $v_j$'s to get $v_j =0$ for $j = -(m-1), -(m-2) , \cdots, (m-2), (m-1)$ which also proves $f=0$
%   Assume  Let $f \in \Lambda_k$ for some fix $ -m \leq k \leq m$. We want to show $f =0$ if $I_a(f) =0$. This $f$ can be think of as an special case symmetric $m$-tensor on $SM$. Therefore $s$-injectivity of $I_a$ implies the existence of an $(m-1)$-tensor $v$ such that $(X+a)v =-f$. 
%%   Let us think $e^{ik\theta} f$ as the restriction of the symmeric $m$-tensor field $e^{m \lambda} f(x) ( dz^{m-k} \otimes d \bar{z}^k)$  Let $f$ be such that $I_a(e^{ik\theta} f) =0$
%%   
%Let $f \in \Omega_k$ and $I_a(f) = 0$  then we want to show that $ f=0$. By $s$-injectivity there exists an $(m-1)$-tensor field $v$ such that $ (X+a)v = -f$.  
\end{proof}

\subsection{Proof of Theorem \ref{thm:transverse}}

The proof of Theorem \ref{thm:transverse} is mainly based on symmetric tensor algebra. To fix ideas, recall that given a 2-dimensional vector space $(V,\dprod{\cdot}{\cdot})$ with basis $\{v_1, v_2\}$, the space $S^m V$ of symmetric $m$-tensors is spanned by 
$$\{\sigma(v_1^{m-q}\otimes v_2^{k}): 0\leq q \leq m \}, $$
with $\sigma$ the symmetrization operator defined as follows:
$$ \sigma ( v_1 \otimes \dots \otimes v_m) = \frac{1}{m !}\sum_{\pi \in \mathfrak{S}_m} v_{\pi(1)} \otimes v_{\pi(2)} \otimes \cdots \otimes v_{\pi(m)},$$
where $\mathfrak{S}_m$ is the permutation group of order $m$. A natural inner product on $S^m V$ is given via the {\em permanent} of the Gram matrix, namely,
\begin{align*}
    \left\langle \sigma(v_1 \otimes \cdots \otimes v_m),\sigma(u_1 \otimes \cdots \otimes u_m) \right\rangle_{S^m V} &= \text{ per } \left( \{a_{ij} \}_{1\le i,j\le m}  \right) := \sum_{\pi \in  \mathfrak{S}_m} a_{ 1 \pi(1)} \cdots  a_{ m \pi(m)},
\end{align*}
where $a_{ij} := \langle u_i, v_j\rangle$, see \cite{Ra96}.

%\begin{definition}[Permanent of a matrix]\label{Def:Permanent}
%For $2m$ vectors $v_1 , \dots , v_m, u_1, \dots , u_m$ in some inner product space $(V, \langle\cdot,\cdot\rangle)$ define an $m \times m$ matrix $A$ whose $ij$-th element is given by $\langle v_i, u_j\rangle $. Then we define the permanent of $A$ as the following:

% of dimension $n$ and let  Denote $A$ Then we define the permanent of the $m \times m$-matrix $(\langle v_i, u_j\rangle )_{1 \leq i,j \leq m}$
%\begin{align}\label{eq:definition of permanent}
%\mbox{per } A =m! \langle \sigma(v_1 \otimes \cdots \otimes v_m),\sigma(u_1 \otimes \cdots \otimes u_m)\rangle.
%\end{align}
%\end{definition}
%\begin{remark}
%In general the permanent of an $(m \times m)$-matrix $A = (a_{ij})$ is defined as 
%$$ per \ A = \sum_{\pi \in  \mathfrak{S}_m} a_{ 1 \pi(1)} \cdots  a_{ m \pi(m)}. $$
%\end{remark}

\begin{proof}[Proof of Theorem \ref{thm:transverse}] Without loss of generality, we work on a global chart of isothermal coordinates, where $g = e^{2\lambda} (dx^2 + dy^2)$, and where a unit tangent vector takes the form $v = e^{-\lambda} (\cos \theta \partial_x + \sin\theta \partial_y)$. A symmetric $m$-tensor $f$ decomposes in the local basis $T_{m,q}:= \sigma (dz^{m-q}\otimes d\bar{z}^q)$ with $0\le q\le m$, in the form 
    \begin{align*}
	f = \sum_{q=0}^m f_q(x) T_{m,q}, \qquad f_q\in L^2(M).
    \end{align*}
    In particular, we have 
    \begin{align*}
	\ell_{m,p} f(x,v) = \sum_{q=0}^m f_q(x) \dprod{T_{m,q}}{v^{m-p}\otimes (v^\perp)^p}.
    \end{align*}
    Now with $v$ expressed in isothermal coordinates, we have
    \begin{align*}
	\dprod{dz}{v} = e^{-\lambda} e^{i\theta}, \quad \dprod{dz}{v^\perp} = e^{-\lambda} i e^{i\theta}, \quad \dprod{d\zbar}{v} = e^{-\lambda} e^{-i\theta}, \quad \dprod{d\zbar}{v^\perp} = -i e^{-\lambda} e^{-i\theta}. 
    \end{align*}
    In particular, we can see that for any vector of the form $w = \alpha v + \beta v^\perp$, 
    \begin{align*}
	\dprod{dz}{w} = e^{-\lambda} e^{i\theta} \dprod{\binom{1}{i}}{\binom{\alpha}{\beta}}, \qquad \dprod{d\zbar}{w} =  e^{-\lambda} e^{-i\theta} \dprod{\binom{1}{-i}}{\binom{\alpha}{\beta}}.
    \end{align*}
    Then, using the definition directly, and using the notation $\zeta\mapsto \zeta'$ to denote the isomorphism mapping $v\mapsto \binom{1}{0}$ and $v^\perp\mapsto \binom{0}{1}$, we arrive at the fact that 
    \begin{align*} 
	\dprod{T_{m,q}}{v^{m-p}\otimes (v^\perp)^p} &= \left\langle dz^{m-q} \otimes d\zbar^{q}, \sigma(v_1\otimes v_2\otimes \cdots \otimes v_m )\right\rangle, \quad \text{where } (m-p)\  v_k\text{'s} \text{ are }v \\
	&= \frac{1}{m!} \sum_{\pi \in \mathfrak{S}_m}\prod_{k=1}^{m-q}\langle dz , v_{\pi(k)}\rangle \prod_{j=m-q+1}^m\langle
	d\zbar , v_{\pi(j)}\rangle\\
	&= \frac{e^{-m\lambda} e^{i(m-2q) \theta}}{m!} \sum_{\pi \in \mathfrak{S}_m}\prod_{k=1}^{m-q} \dprod{\binom{1}{i}}{v^\prime_{\pi(k)}} \prod_{j=m-q+1}^m\left\langle \binom{1}{-i}, v^\prime_{\pi(j)}\right\rangle\\
	&= e^{-m\lambda} e^{i(m-2q) \theta} \left\langle \binom{1}{i}^{m-q} \otimes \binom{1}{-i}^q, \sigma(v^{\prime}_{1}\otimes \cdots \otimes v^\prime_{m})\right\rangle\\
	&= e^{-m\lambda} e^{i(m-2q) \theta}\left\langle \sigma\left(\begin{pmatrix}
	    1\\i
	\end{pmatrix}^{m-q} \otimes \binom{1}{-i}^q\right), \begin{pmatrix}
	    1\\0
	\end{pmatrix}^{m-p} \otimes \begin{pmatrix}
	    0\\1
	\end{pmatrix}^p\right\rangle \\
	&= A^{(m)}_{pq} e^{-m\lambda} e^{i(m-2q)\theta}, 
    \end{align*}
    where the matrix $A^{(m)} \in M_{m+1}(\Cm)$ with elements
    \begin{align}
	A_{pq}^{(m)} := \left\langle \sigma\left(\begin{pmatrix}
	    1\\i
	\end{pmatrix}^{m-q} \otimes \binom{1}{-i}^q\right), \begin{pmatrix}
	    1\\0
	\end{pmatrix}^{m-p} \otimes \begin{pmatrix}
	    0\\1
	\end{pmatrix}^p\right\rangle, \qquad 0\le p,q\le m,
	\label{eq:Ampq}
    \end{align}
    is a constant matrix. In particular, this yields the relation
    \begin{align*}
	\ell_{m,p} f(x,v) = \sum_{q=0}^m A^{(m)}_{pq} e^{-m\lambda(x)}f_q(x) e^{i(m-2q)\theta},
    \end{align*}
    and hence 
    \begin{align*}
	\begin{pmatrix}
	    I_{w,m} [\ell_{m,0} f] \\
	    I_{w,m} [\ell_{m,1} f] \\	    
	    \vdots \\
	    I_{w,m} [\ell_{m,m} f]
	\end{pmatrix}
	= A^{(m)} 
	\begin{pmatrix}
	    I_w [e^{-m\lambda} f_0 e^{im\theta}] \\ I_w [e^{-m\lambda} f_1 e^{i(m-2)\theta}] \\ \vdots\\ I_w [e^{-m\lambda} f_m e^{-im\theta}].
	\end{pmatrix}
    \end{align*}
    Therefore, the theorem is proved if $A^{(m)}$ is invertible, and we now explain how to invert $A^{(m)}$ explicitly. 

    Since $\sigma$ is self-adjoint and idempotent, we may rewrite 
    \begin{align*}
	A_{pq}^{(m)} := \dprod{ \sigma\left( \binom{1}{0}^{m-p} \otimes \binom{0}{1}^p \right) }{  \sigma\left( \binom{1}{i}^{m-q} \otimes \binom{1}{-i}^q\right) }, \qquad 0\le p,q\le m,
    \end{align*}
    so $A^{(m)}$ is a matrix of inner products between two bases of $S^m (\Cm^2)$ (of complex dimension $m+1$) endowed with the {\em permanent} as inner product. Each is a basis because it's an orthogonal system, and moreover, we compute their norms by direct calculation
    \begin{align*}
	\dprod{ \sigma\left( \binom{1}{i}^{m-q}\otimes \binom{1}{-i}^q \right) }{ \sigma\left( \binom{1}{i}^{m-q}\otimes \binom{1}{-i}^q \right) } &= 2^m \binom{m}{q}, \\ 
	\dprod{ \sigma\left( \binom{1}{0}^{m-p}\otimes \binom{0}{1}^p \right) }{ \sigma\left( \binom{1}{0}^{m-p}\otimes \binom{0}{1}^p \right) } &= \binom{m}{p}.
    \end{align*}
    In particular, the matrix
    \begin{align*}
	B^{(m)} = \text{diag } \left( \binom{m}{p}^{-\frac{1}{2}},\ 0\le p\le m  \right) A^{(m)} \text{diag }\left( 2^{-m/2} \binom{m}{q}^{-\frac{1}{2}},\ 0\le q\le m \right)
    \end{align*}
    is unitary as it amounts to the inner products of two orthonormal bases so $(B^{(m)})^{-1} = (B^{(m)})^*$, and this directly yields an inverse for $A^{(m)}$, given by 
    \begin{align*}
	(A^{(m)})^{-1} = \text{diag }\left( 2^{-m/2} \binom{m}{q}^{-\frac{1}{2}},\ 0\le q\le m \right) (B^{(m)})^* \text{diag } \left( \binom{m}{p}^{-\frac{1}{2}},\ 0\le p\le m  \right),
    \end{align*}
    of general term 
    \begin{align*}
	(A^{(m)})_{pq}^{-1} = 2^{-m/2} \binom{m}{p}^{-\frac{1}{2}} \binom{m}{q}^{-\frac{1}{2}} \overline{A_{qp}^{(m)}}, \qquad 0\le p,q \le m. 
    \end{align*} 
\end{proof}

\section*{Acknowledgements}

V.P.K is partially supported by NSF grant DMS 1616564 and a SERB Matrics grant. F.M. is partially funded by NSF grant DMS-1814104 and a Hellmann Fellowship. The authors would like to thank Gunther Uhlmann and the Institute of Advanced Studies of the Hong Kong University of Science and Technology for their hospitality, as this work was initiated during the ``Inverse Problems, Imaging and Partial Differential Equations'' workshop organized there in December 2016.

%\bibliographystyle{siam}
%\bibliography{./bibliography}

\begin{thebibliography}{9}

\bibitem{Abhishek2017}
A.~Abhishek and R.~K. Mishra, 
{\em Support theorems and an injectivity result for integral moments of a symmetric m-tensor field}, 
J Fourier Anal Appl (2018). https://doi.org/10.1007/s00041-018-09649-7

\bibitem{Assylbekov2017}
Y.~M. Assylbekov, F.~Monard, and G.~Uhlmann, 
{\em Inversion formulas and range characterizations for the attenuated geodesic ray transform}, 
Journal de Math\'ematiques Pures et Appliqu\'ees 111 (March 2018), 161--190.

\bibitem{Ra96}
R.~Bhatia, 
{\em Matrix Analysis}, 
Springer, New York, 1996.


\bibitem{Derevtsov1} 
E.~Yu.~Derevtsov, I.~E.~Svetov, Yu.~S.~Volkov and T.~Schuster, 
{\em Numerical $B$-spline solution of emission and vector 2D-tomography problems for media with absorption and refraction}, 
Proceedings 2008 IEEE Region 8 International Conference on Computational Technologies in Electrical and Electronics Engineering SIBIRCON-08, Novosibirsk Scientific Center, Novosibirsk, Russia, July 21-25, 2008, pp. 212-217.

\bibitem{Derevtsov2}
E.~Yu.~Derevtsov, V.~V.~Pickalov, 
{\em Reconstruction of vector fields and their singularities from ray transform}, 
Numerical Analysis and Applications, 2011, Vol. 4, No. 1, pp. 21-35.

\bibitem{Derevtsov3}
E.~Yu.~Derevtsov, I.~E.~Svetov, 
{\em Tomography of tensor fields in the plane}, 
Eurasian J. Math. Comp. Applications, 2015, Vol. 3, No. 2, pp. 24-68.

\bibitem{Derevtsov4}
E.~Yu.~Derevtsov, S.~V.~Maltseva, 
{\em Reconstruction of the Singular Support of a Tensor Field Given in a Refracting Medium by Its Ray Transform}, 
Journal of Applied and Industrial Mathematics, 2015, Vol. 9, No. 4, pp. 447-460.


\bibitem{Guillemin1980}
V.~Guillemin and D.~Kazhdan, 
{\em Some inverse spectral results for negatively curved 2-manifolds}, 
Topology, 19 (1980), pp.~301--312.

\bibitem{Hammer2004}
H.~Hammer and B.~Lionheart, 
{\em Application of Sharafutdinov's ray transform in integrated photoelasticity}, 
Journal of Elasticity, 75 (2004), pp.~229--246.

\bibitem{Ilmavirta2017}
J.~Ilmavirta, J.~Lehtonen, and M.~Salo, 
{\em Geodesic x-ray tomography for piecewise constant functions on nontrapping manifolds}, 
Mathematical
  Proceedings of the Cambridge Philosophical Society (to appear),  (2017).

\bibitem{Ilmavirta2018}
J.~Ilmavirta and F.~Monard, 
{\em Integral geometry on manifolds with boundary and applications}, 
Inverse Problems (to appear),  (2018).
\newblock arXiv:1806.06088.

\bibitem{Kazantsev2004}
S.~G.~ Kazantsev and A.~A.~Bukhgeim, 
{\em Singular value decomposition for the 2d fan-beam radon transform of tensor fields}, 
Journal of Inverse and Ill-posed Problems, 12 (2004), pp.~245--278.

\bibitem{KMSS}
V.~P.~Krishnan, R.~Manna, S.~K.~Sahoo, and V.~A.~Sharafutdinov, 
{\em Momentum ray transforms}, Inverse Problems and Imaging, to appear (2018). 

\bibitem{Monard2013a}
F.~Monard, 
{\em On reconstruction formulas for the {X}-ray transform acting on symmetric differentials on surfaces}, Inverse Problems, 30 (2014), p.~065001.

\bibitem{Monard2015a}
F.~Monard,  
{\em Efficient tensor tomography in fan-beam coordinates}, 
Inverse Probl. Imaging, 10 (2016), pp.~433--459.
\newblock %arXiv:1510.05132.

\bibitem{Monard2015}
F.~Monard, 
{\em Inversion of the attenuated geodesic {X}-ray transform over functions and vector fields on simple surfaces}, 
SIAM J. Math. Anal., 48 (2016), pp.~1155--1177.
\newblock %arXiv:1503.07190.

\bibitem{Monard2017a}
F.~Monard, 
{\em Efficient tensor tomography in fan-beam coordinates. {II}: attenuated transforms}, 
 Inverse Problems and Imaging, 12:2 (2018), 433-460.

\bibitem{Palamodov2015}
V.~Palamodov, 
{\em On reconstruction of strain fields from tomographic data}, 
Inverse Problems, 31 (2015), p.~085002.

\bibitem{Paternain2013a}
G.~P.~Paternain, M.~Salo, and G.~Uhlmann, 
{\em On the range of the attenuated ray transform for unitary connections}, 
International Math. Research Notices, (online) (2013).
\newblock %arXiv:1302.4880.

\bibitem{Paternain2011a}
G.~P.~Paternain, M.~Salo, and G.~Uhlmann,  
{\em Tensor tomography on surfaces}, 
Inventiones Math., 193 (2013), pp.~229--247.
\newblock %arXiv:1109.0505.

\bibitem{Paternain2015}
G.~P.~Paternain, M.~Salo, and G.~Uhlmann, 
{\em Invariant distributions, Beurling transforms and tensor tomography in higher dimensions},
Mathematische Annalen, 363 (2015), pp.~305--362.

\bibitem{Salo2011}
M.~Salo and G.~Uhlmann, 
{\em {T}he {A}ttenuated {R}ay {T}ransform on {S}imple {S}urfaces}, 
J. Diff. Geom., 88 (2011), pp.~161--187.

\bibitem{Sharafutdinov1986}
V.~A.~Sharafutdinov,
{\em A problem of integral geometry for generalized tensor fields on $\mathbb{R}^n$}, 
Soviet Math. Dokl, vol.~33, 1986, pp.~100--102.

\bibitem{Sharafudtinov1994}
V.~A.~Sharafutdinov, 
{\em Integral geometry of tensor fields}, 
{VSP}, Utrecht, The Netherlands, 1994.

\bibitem{Sharafudtinov1997}
V.~A.~Sharafutdinov, 
{\em Integral geometry of tensor fields on a surface of revolution}, 
Siberian Math. J., 38 (1997).

\bibitem{Svetov2013}
I.~E.~Svetov, E.~Y.~Derevtsov, Y.~S.~Volkov, and T.~Schuster, 
{\em A numerical solver based on {B}-splines for 2{D} vector field tomography in a refracting medium.}, 
Mathematics and Computers in Simulation, 94 (2013), pp.~15--32.


%\bibitem{bonnafe:2008}
%C. Bonnaf\'e and M. J. Dyer,
%\emph{Semidirect product decomposition of Coxeter groups}, preprint (2008),
%\url{http://arxiv.org/abs/0805.4100}.

%\bibitem{ciet:2003}
%M. Ciet, T. Lange, F. Sica and J.-J. Quisquater,
%Improved algorithms for efficient arithmetic on elliptic curves using fast endomorphisms,
%in: \emph{Advances in Cryptology},
%Lecture Notes in Math. 2656, Springer, Berlin (2003), 388--400.

%\bibitem{king:1995}
%J. D. King,
%\emph{Finite presentability of Lie algebras and pro-$p$ groups},
%Ph.D. thesis, University of Cambridge, 1995.

%\bibitem{lange:2007}
%T. Lange and I. E. Shparlinski,
%Distribution of some sequences of points on elliptic curves,
%\emph{J. Math. Cryptol.} \textbf{1} (2007), 1--11.

%\bibitem{schmitt:2003}
%S. Schmitt and H. G. Zimmer,
%\emph{Elliptic Curves -- A Computational Approach},
%2nd ed., de Gruyter Stud. Math. 31, De Gruyter, Berlin, 2003.

\end{thebibliography}

%\input moment.bbl

\end{document}